\documentclass[11pt]{amsart}

\usepackage[usenames]{color}

\usepackage{amsmath,times,pslatex,enumitem}
\usepackage{amsthm}
\usepackage{amsfonts}
\usepackage{amssymb}
\usepackage{latexsym}

\newtheorem{theorem}{Theorem}
\newtheorem{corollary}{Corollary}

\newtheorem{lemma}{Lemma}

\newtheorem{proposition}{Proposition}

\newcommand{\Z}{\mathbb{Z}}

\newcommand{\F}{\mathbb{F}}

\newcommand{\athree}{A(3)}
\newcommand{\bthree}{B(3)}
\newcommand{\cthree}{C(3)}

\newcommand{\cthreen}{(\cthree)^n}

\title{Navigating in the Cayley graph of $SL_2(\F_p)$\\
 and applications to hashing}

\author[]{Lisa Bromberg}
\address{Graduate Center, City University of New York}
\email{lisa.bromberg@gmail.com}

\author[]{Vladimir Shpilrain}
\address{Department of Mathematics, The City  College  of New York, New York,
NY 10031} \email{shpil@groups.sci.ccny.cuny.edu}
\thanks{Research of the second author was partially supported by
the NSF grant CNS-1117675}

\author[]{Alina Vdovina}\address{School of Mathematics and
Statistics, University of Newcastle, Newcastle upon Tyne, NE1 7RU,
U.K.} \email{alina.vdovina@ncl.ac.uk}

\begin{document}

\maketitle

\begin{abstract}

Cayley hash functions are based on a simple idea of using a pair of
(semi)group elements,  $A$ and  $B$, to hash the 0 and 1 bit,
respectively, and  then to hash an arbitrary bit string in the
natural way, by using multiplication of elements in the (semi)group.
In this paper, we focus on hashing with $2 \times 2$ matrices over
$\F_p$. Since there are many known pairs of $2 \times 2$ matrices
over $\Z$ that generate a free monoid, this yields numerous pairs of
matrices over $\F_p$, for a sufficiently large prime $p$, that are
candidates for collision-resistant hashing. However, this trick can
``backfire", and lifting matrix entries to $\Z$ may facilitate
finding a collision. This ``lifting attack" was successfully used by
Tillich and Z\'emor in the special case where two matrices $A$ and
$B$ generate (as a monoid) the whole monoid $SL_2(\Z_+)$. However,
in this paper we show that the situation with other, ``similar",
pairs of matrices from $SL_2(\Z)$ is different, and the ``lifting
attack" can (in some cases) produce collisions in the {\it group}
generated by $A$ and $B$, but not in the positive {\it monoid}.
Therefore, we argue that for these pairs of matrices, there are no
known attacks at this time that would affect security of the
corresponding hash functions. We also give explicit lower bounds on
the length of collisions for hash functions corresponding to some
particular pairs of matrices from $SL_2(\F_p)$.

\end{abstract}

\section{Introduction}

Hash functions are easy-to-compute compression functions  that take
a variable-length input and convert it to a fixed-length output.
Hash functions are used as compact representations, or digital
fingerprints, of data and to provide message integrity. Basic
requirements are well known:

\begin{enumerate}

  \item {\it Preimage resistance} (sometimes called {\it non-invertibility}): it should be computationally
infeasible to find an input which hashes to a specified output;

\item {\it Second pre-image resistance}: it
should be computationally infeasible to find a second input that
hashes to the same output as a specified input;

\item {\it Collision resistance}: it should
be computationally infeasible to find two different inputs that hash
to the same output.

\end{enumerate}

A challenging problem is to determine mathematical properties of a
hash function that would ensure (or at least, make it likely) that
the requirements  above are met.

Early suggestions (especially the SHA family) did not really use any
mathematical ideas apart from the Merkle-Damgard construction for
producing collision-resistant hash functions from
collision-resistant compression functions (see e.g.~\cite{MOV}); the
main idea was just to ``create a mess" by using complex iterations
(this is not meant in a derogatory sense, but just as an opposite of
using mathematical structure one way or another).

An interesting direction worth mentioning is constructing hash
functions that are provably as secure as underlying assumptions,
e.g.~as discrete logarithm assumptions; see \cite{CLS} and
references therein. These hash functions however tend to be not very
efficient. For a general survey on hash functions we refer to
\cite{MOV}.

Another direction, relevant to the present paper, is using a pair of
elements, $A$ and  $B$, of a semigroup $S$, such that the Cayley
graph of the semigroup generated by $A$ and  $B$ is expander, in the
hope that such a graph would have a large girth and therefore there
would be no short relations. Probably the most popular
implementation of this idea so far is the Tillich-Z\'emor hash
function \cite{TZ}. We refer to \cite{Petit_thesis} and
\cite{Petit_Rubik} for a more detailed survey on Cayley hash
functions.

The Tillich-Z\'emor hash function, unlike functions in the SHA
family, is {\it not} a block hash function, i.e., each bit is hashed
individually. More specifically, the ``0" bit is hashed to a
particular $2\times 2$ matrix $A$, and the ``1" bit is hashed to
another $2\times 2$ matrix $B$. Then a bit string is hashed  simply
to the product of matrices $A$ and $B$ corresponding to bits in this
string. For example, the bit string 1000110 is hashed to the matrix
$BA^3B^2A$.

Tillich and Z\'emor use matrices $A$, $B$ from the group $SL_2(R)$,
where $R$ is a commutative ring (actually, a field) defined as
$R={\mathbf F}_2[x]/(p(x))$. Here ${\mathbf F}_2$ is the field with
two elements,  ${\mathbf F}_2[x]$ is the ring of polynomials over
${\mathbf F}_2$,  and $(p(x))$ is the ideal of ${\mathbf F}_2[x]$
generated by   an irreducible polynomial $p(x)$ of degree $n$
(typically,  $n$ is a prime, $127 \le n \le 170$); for example,
$p(x)=x^{131}+x^7+x^6+x^5+x^4+x+1$.
 Thus, $R={\mathbf F}_2[x]/(p(x))$ is isomorphic to ${\mathbf F}_{2^n}$, the field with
 $2^n$ elements.

Then, the matrices  $A$ and  $B$ are:

\begin{displaymath}
A = \left(
 \begin{array}{cc} \alpha & 1 \\ 1 & 0 \end{array} \right) , \hskip 1cm B = \left(
 \begin{array}{cc} \alpha & \alpha+1 \\ 1 & 1 \end{array} \right),
\end{displaymath}

\noindent where $\alpha$ is a root of $p(x)$.
\smallskip

Another idea of the same kind is to use a pair of $2 \times 2$
matrices, $A$ and $B$, over $\Z$  that generate a free monoid, and
then reduce the entries modulo a large prime $p$ to get matrices
over $\F_p$. Since there cannot be an equality of two different
products of positive powers of $A$ and $B$ unless at least one of
the entries in at least one of the products is $\ge p$, this gives a
lower bound on the minimum length of bit strings where a collision
may occur. This lower bound is going to be on the order of $\log p$;
we give more precise bounds for some particular examples of $A$ and
$B$ in our Section \ref{Girth}.

The first example of a pair of matrices over $\Z$  that generate a
free monoid is:

\begin{displaymath}
A(1) = \left(
 \begin{array}{cc} 1 & 1 \\ 0 & 1 \end{array} \right) , \hskip 1cm B(1) = \left(
 \begin{array}{cc} 1 & 0 \\ 1 & 1 \end{array} \right).
\end{displaymath}

These matrices are obviously invertible, so they actually generate
the whole group $SL_2(\Z)$. This group is not free, but the {\it
monoid} generated by $A(1)$ and $B(1)$ is free, and this is what
matters for hashing because only positive powers of $A(1)$ and
$B(1)$ occur in hashing.  However, the fact that these two matrices
generate the whole  $SL_2(\Z)$ yields an attack on the corresponding
hash function (where the matrices $A(1)$ and $B(1)$ are considered
over $\mathbf F_p$, for a large $p$), see \cite{TZ_attack}, where a
collision is found by using Euclidean algorithm on the entries of a
matrix.

At this point, we note that a pair of matrices

\begin{displaymath}
A(x) = \left(
 \begin{array}{cc} 1 & x \\ 0 & 1 \end{array} \right) , \hskip 1cm B(y) = \left(
 \begin{array}{cc} 1 & 0 \\ y & 1 \end{array} \right)
\end{displaymath}

\noindent generate a free subgroup of $SL_2(\Z)$ if $xy \ge 4$.

In Section \ref{attack}, we consider the following pair of matrices:

\begin{displaymath}
A(2) = \left(
 \begin{array}{cc} 1 & 2 \\ 0 & 1 \end{array} \right) , \hskip 1cm B(2) = \left(
 \begin{array}{cc} 1 & 0 \\ 2 & 1 \end{array} \right).
\end{displaymath}

By using a result from an old paper of Sanov \cite{Sanov} and
combining it with the attack on hashing with   $A(1)$ and $B(1)$
offered in \cite{TZ_attack}, we show that there is an  efficient
heuristic algorithm that finds circuits of length $O(\log p)$ in the
Cayley graph of the {\it group} generated by $A(2)$ and $B(2)$,
considered as matrices over $\mathbf F_p$. However, this has no
bearing on the (in)security of the hash function based on $A(2)$ and
$B(2)$ since in hashing only positive powers of $A(2)$ and $B(2)$
are used, and group relations of length $O(\log p)$ produced by the
mentioned algorithm will involve negative as well as positive powers
with overwhelming probability.

To conclude the Introduction, we mention that the fact that Cayley
graphs of the groups generated by $A(1), B(1)$ and by $A(2), B(2)$
are expanders was known for a couple of decades (see e.g. \cite{Lub}
for all background facts on expanders), but the same property of the
Cayley graph relevant to $A(3), B(3)$ was a famous 1-2-3 question of
Lubotzky, which was settled in the positive in \cite{BG} using
results from \cite{H}.

We also give a word of caution: while intuitively, we think of
expander graphs as graphs with a large (``expanding") girth, this is
not necessarily the case in general. In particular, first explicit
examples of expander graphs due to Margulis\cite{Marg}  have bounded
girth.

\section{Hashing with $A(2)$ and $B(2)$ and  circuits in the Cayley graph}
\label{attack}

In this section, motivated by hashing with the  matrices $A(2)$ and
$B(2)$ considered as matrices over $\mathbf F_p$,  we discuss
circuits in the relevant Cayley graph.

Tillich and Z\'emor \cite{TZ_attack} offered an attack on the  hash
function based on $A(1)$ and $B(1)$ (again, considered as matrices
over $\mathbf F_p$). To the best of our knowledge, this is the only
published attack on that hash function. In this section we explain
why this particular attack should not work with the matrices $A(2)$
and $B(2)$, and this therefore leaves the door open for using these
matrices (over $\mathbf F_p$, for a sufficiently large $p$) for
hashing.

First we explain, informally, what appears to be the reason why the
attack from \cite{TZ_attack} should not work with $A(2)$ and $B(2)$.
The reason basically is that, while $A(1)$ and $B(1)$ (considered
over $\Z$) generate (as a monoid!) the whole monoid of $2 \times 2$
matrices over $\Z$ with positive entries, with the matrices $A(2)$
and $B(2)$ the situation is much less transparent. There is a result
from an old paper by Sanov \cite{Sanov} that says: the {\it
subgroup} of $SL_2(\Z)$ generated by $A(2)$ and $B(2)$ consists of
{\it all} matrices of the form $\left( \begin{array}{cc} 1+4m_1 &
2m_2 \\ 2m_3 & 1+m_4\end{array} \right),$ where all $m_i$ are
arbitrary integers. This, however, does not tell much about the {\it
monoid} generated by $A(2)$ and $B(2)$. In fact, a generic matrix of
the above form would {\it not} belong to this monoid. This is not
surprising because: (1) $A(2)$ and $B(2)$ generate a free group, by
another result of Sanov \cite{Sanov}; (2) the number of different
elements represented by {\it all} freely irreducible words in $A(2)$
and $B(2)$ of length $m \ge 2$ is $4 \cdot 3^{m-1}$, whereas  the
number of different elements represented by {\it positive} words of
length $m \ge 2$ is $2^m$. Thus, the share of matrices in the above
form representable by positive words in $A(2)$ and $B(2)$ is
exponentially negligible.

What Tillich and Z\'emor's ``lifting attack" \cite{TZ_attack} can
still give is an efficient heuristic algorithm that finds relations
of length $O(\log p)$ in the {\it group} generated by $A(2)$ and
$B(2)$, considered as matrices over $\mathbf F_p$. We describe this
algorithm below because we believe it might be useful in other
contexts, although it has no bearing on the security of the hash
function based on $A(2)$ and $B(2)$ since in hashing only positive
powers of $A(2)$ and $B(2)$ are used, and group relations of length
$O(\log p)$ produced by the algorithm mentioned above will involve
negative as well as positive powers with overwhelming probability,
even if $p$ is rather small. What one would need to attack the hash
function corresponding to $A(2)$ and $B(2)$ is a result, similar to
Sanov's, describing all matrices in the {\it monoid} generated by
$A(2)$ and $B(2)$.

We are now going to use a combination of the attack on  $A(1)$ and
$B(1)$ offered in \cite{TZ_attack} with the aforementioned result of
Sanov \cite{Sanov} to find relations of length $O(\log p)$ in the
{\it group} generated by $A(2)$ and $B(2)$. 

\begin{theorem}
There is an efficient heuristic algorithm that finds particular
relations of the form $w(A(2), B(2))=1$, where $w$ is a group word
of length $O(\log p)$, and the matrices $A(2)$ and $B(2)$ are
considered over $\F_p$.
\end{theorem}

\begin{proof}
It was shown in \cite{TZ_attack} that:

\noindent {\bf (a)} For any prime $p$, there is an efficient
heuristic algorithm that finds positive integers $k_1, k_2, k_3,
k_4$ such that the matrix $\left( \begin{array}{cc} 1+k_1p & k_2p \\
k_3p & 1+k_4p\end{array} \right)$ has determinant 1 and all $k_i$
are of about the same magnitude  $O(p^2)$.
\medskip

\noindent {\bf (b)} A generic matrix from part (a) has an efficient
factorization (in $SL_2(\Z)$) in a product of positive powers of
$A(1)$ and $B(1)$, of length $O(\log p)$. (This obviously yields a
collision in $SL_2({\mathbf F}_p)$ since the matrix from part (a)
equals the identity matrix in $SL_2({\mathbf F}_p)$.)
\medskip

Now we combine these results with the aforementioned result of Sanov
the following way. We are going to multiply a matrix from (a) (call
it $M$) by a matrix from $SL_2(\Z)$  (call it $S$) with very small
(between 0 and 5 by the absolute value) entries, so that the
resulting matrix $M \cdot S$ has the form as in Sanov's result.
Since the matrix $M$, by the Tillich-Z\'emor results, has an
efficient factorization (in $SL_2(\Z)$) in a product $w(A(1), B(1))$
of  powers of $A(1)$ and $B(1)$ of length $O(\log p)$, the same
holds for the matrix $M \cdot S$. Then, since the matrix $M \cdot S$
is in ``Sanov's form", we know that it is, in fact, a product of
powers of $A(2)$ and $B(2)$.

Now we need one more ingredient to efficiently re-write a product of
$A(1)$ and $B(1)$ into a product of $A(2)$ and $B(2)$ without
blowing up the length too much.  This procedure is provided by
Theorem 2.3.10  in \cite{Epstein}. We cannot explain it without
introducing a lot of background material, but the fact is that,
since the group $SL_2(\Z)$ is hyperbolic    (whatever that means)
and the subgroup generated by $A(2)$ and $B(2)$ is quasiconvex
(whatever that means), there is a quadratic time algorithm (in the
length of the word $w(A(1), B(1))$) that re-writes $w(A(1), B(1))$
into a $u(A(2), B(2))$ such that $w(A(1), B(1)) = u(A(2), B(2))$ and
the length of $u$ is bounded by a constant (independent of $w$)
times the length of $w$.

Thus, what is now left to complete the proof is to exhibit, for all
possible matrices $M$ as in part (a) above, particular ``small"
matrices $S$ such that $M \cdot S$ is in ``Sanov's form". We are
therefore going to consider many different cases corresponding to
possible combinations of residues    modulo 4 of the entries of the
matrix $M$ (recall that $M$ has to have determinant 1), and in each
case we are going to exhibit the corresponding matrix $S$ such that
$M \cdot S$ is in ``Sanov's form". Denote by $\hat M$ the matrix of
residues modulo 4 of the entries of $M$.  Since the total number of
cases is too large, we consider matrices $\hat M$ ``up to a
symmetry".

\noindent {\bf (1)} $\hat M = \left(
 \begin{array}{cc} 1 & 0 \\ 1 & 1 \end{array} \right)$, ~$S = \left(
 \begin{array}{cc} 1 & 0 \\ 1 & 1 \end{array} \right)$

\noindent {\bf (2)} $\hat M = \left(
 \begin{array}{cc} 1 & 0 \\ 2 & 1 \end{array} \right)$, ~$S = \left(
 \begin{array}{cc} 1 & 0 \\ 2 & 1 \end{array} \right)$

\noindent {\bf (3)} $\hat M = \left(
 \begin{array}{cc} 1 & 0 \\ 3 & 1 \end{array} \right)$, ~$S = \left(
 \begin{array}{cc} 1 & 0 \\ 3 & 1 \end{array} \right)$

\noindent {\bf (4)} $\hat M = \left(
 \begin{array}{cc} 2 & 1 \\ 1 & 1 \end{array} \right)$, ~$S = \left(
 \begin{array}{cc} 1 & 3 \\ 3 & 2 \end{array} \right)$

\noindent {\bf (5)} $\hat M = \left(
 \begin{array}{cc} 2 & 3 \\ 3 & 1 \end{array} \right)$, ~$S = \left(
 \begin{array}{cc} 3 &  1 \\ 1 & 2 \end{array} \right)$

\noindent {\bf (6)} $\hat M = \left(
 \begin{array}{cc} 2 & 3 \\ 1 & 2 \end{array} \right)$, ~$S = \left(
 \begin{array}{cc} 0 & 1 \\ 3 & 2 \end{array} \right)$

\noindent {\bf (7)} $\hat M = \left(
 \begin{array}{cc} 2 & 1 \\ 1 & 3 \end{array} \right)$, ~$S = \left(
 \begin{array}{cc} 1 & 3 \\ 3 & 2 \end{array} \right)$

\noindent {\bf (8)} $\hat M = \left(
 \begin{array}{cc} 2 & 3 \\ 3 & 3 \end{array} \right)$, ~$S = \left(
 \begin{array}{cc} 1 &  1 \\ 1 & 2 \end{array} \right)$

\noindent {\bf (9)} $\hat M = \left(
 \begin{array}{cc} 3 & 3 \\ 0 & 1 \end{array} \right)$, ~$S = \left(
 \begin{array}{cc} 3 & 3 \\ 0 & 1 \end{array} \right)$

\noindent {\bf (10)} $\hat M = \left(
 \begin{array}{cc} 3 & 0 \\ 0 & 3 \end{array} \right)$, ~$S = \left(
 \begin{array}{cc} -1 & 0 \\ 0 & -1 \end{array} \right)$

\noindent {\bf (11)} $\hat M = \left(
 \begin{array}{cc} 3 & 0 \\ 1 & 3 \end{array} \right)$, ~$S = \left(
 \begin{array}{cc} -1 & 0 \\ 1 & -1 \end{array} \right)$

\noindent {\bf (12)} $\hat M = \left(
 \begin{array}{cc} 3 & 0 \\ 2 & 3 \end{array} \right)$, ~$S = \left(
 \begin{array}{cc} -1 & 0 \\ 2 & -1 \end{array} \right)$

\noindent {\bf (13)} $\hat M = \left(
 \begin{array}{cc} 3 & 0 \\ 3 & 3 \end{array} \right)$, ~$S = \left(
 \begin{array}{cc} -1 & 0 \\ 3 & -1 \end{array} \right)$

\noindent {\bf (14)} $\hat M = \left(
 \begin{array}{cc} 3 & 2 \\ 2 & 3 \end{array} \right)$, ~$S = \left(
 \begin{array}{cc} -1 & 2 \\ 2 & -5 \end{array} \right)$

\medskip

\noindent This completes the proof.

\end{proof}

To conclude this section, we point out an example of re-writing a
word in $A(1)$ and $B(1)$  into a word in  $A(2)$ and $B(2)$. All
matrices here are considered over $\Z$.

$$A(1)B(1)A(1)B(1)A(1)B(1) = A(2)A(2)B(2)^{-1}B(2)^{-1}A(2)^{-1}A(2)^{-1}B(2)B(2).$$

\noindent We see that even in this simple example, both positive and
negative powers of $A(2)$ and $B(2)$ are required.

\section{Girth of the Cayley graph relevant to $A(k)$ and $B(k)$}
\label{Girth}

Our starting point here is the following observation: the entries of
matrices that are products of length $n$ of positive powers of
$A(k)$ and $B(k)$ exhibit the fastest growth (as functions of $n$)
if $A(k)$ and $B(k)$ alternate in the product:
$A(k)B(k)A(k)B(k)\cdots$. More formally:

\begin{proposition}\label{alternate}
Let $w_n(a, b)$ be an arbitrary positive word of even length $n$,
and let $W_n = w_n(A(k), B(k))$, with $k \ge 2$. Let  $C_n = (A(k)
\cdot B(k))^{\frac{n}{2}}$. Then: {\bf (a)} the sum of entries in
any row of $C_n$ is at least as large as the sum of entries in any
row of $W_n$; {\bf (b)} the largest entry of $C_n$ is at least as
large as the largest entry of $W_n$.
\end{proposition}

\begin{proof}
First note that multiplying a matrix $X$ by $A(k)$ on the right
amounts to adding to the second column of $X$ the first column
multiplied by $k$. Similarly, multiplying $X$ by $B(k)$ on the right
amounts to adding to the first column of $X$ the second column
multiplied by $k$. This means, in particular, that when we build a
word in $A(k)$ and  $B(k)$ going left to right, elements of the
first row change independently of elements of the second row.
Therefore, we can limit our considerations to pairs of positive
integers, and the result will follow from the following

\begin{lemma}
Let  $(x, y)$ be a pair of positive integers and let $k \ge 2$. One
can apply transformations of the following two kinds: (1)
transformation $R$ takes $(x, y)$ to  $(x, y+kx)$; ~(2)
transformation $L$ takes $(x, y)$ to  $(x+ky, y)$. Among all
sequences of these transformations of the same length, the sequence
where $R$ and  $L$ alternate results in:  {\bf (a)} the largest sum
of elements in the final pair; {\bf (b)} the largest maximum element
in the final pair.
\end{lemma}

\begin{proof} We are going to prove {\bf (a)} and {\bf (b)} simultaneously using
induction by the length of a sequence of transformations. Suppose
our lemma holds for all sequences of length at most $m \ge 2$, with
the same initial pair $(x, y)$. Suppose the final pair after $m$
alternating transformations is $(X, Y)$. Without loss of generality,
assume that $X<Y$. That means the last applied transformation was
$R$. Now applying $L$ to $(X, Y)$ gives $(X+kY, Y)$, while applying
$R$ to $(X, Y)$ gives $(X, Y+kX)$. Since $X+kY > Y+kX$, applying $L$
results in a larger sum of elements as well as in a larger maximum
element. Thus, we have a sequence of $(m+1)$ alternating
transformations, and now we have to consider one more case.

Suppose some sequence of $m$ transformations applied to $(x, y)$
results in a pair $(X', Y')$ with $X'+Y'<X+Y, ~Y'<Y$, but $X'>X$.
Then applying $L$ to this pair gives $(X'+kY', Y')$, and the sum is
$X'+Y'+kY' <  X+Y+kY$ since $X'+Y'<X+Y$ and  $Y'<Y$. The maximum
element of the pair $(X'+kY', Y')$ is $X'+kY' = X'+Y'+ (k-1)Y'$.
Again, since $X'+Y'<X+Y$ and  $Y'<Y$, we have $X'+kY' < X+kY$. This
completes the proof of the lemma and the proposition.

\end{proof}

\end{proof}

This motivates us to consider powers of the matrix $C(k)=A(k)B(k)$
to get to entries larger than $p$ ``as quickly as possible".

%
%
%

\subsection{Powers of $C(2)=A(2)B(2)$}
The matrix $C(2)$ is $\left(
 \begin{array}{cc} 5 & 2 \\ 2 & 1 \end{array} \right)$. If we
 denote $(C(2))^n = \left(
 \begin{array}{cc} a_n & b_n \\ c_n & d_n \end{array} \right)$, then
 the following recurrence relations are easily proved by induction
 on $n$:

$$a_n = 5a_{n-1} + 2b_{n-1}; \hskip .5cm b_n = c_n = 2a_{n-1} +
b_{n-1}; \hskip .5cm d_n = a_{n-1}.$$

Combining the recurrence relations for $a_n$ and $b_n$, we get $2b_n
= a_n - a_{n-1}$, so $2b_{n-1} = a_{n-1} - a_{n-2}$. Plugging this
into the displayed recurrence relation for $a_n$ gives

$$a_n = 6a_{n-1} - a_{n-2}.$$

Similarly, we get

$$b_n = 6b_{n-1} - b_{n-2}.$$

\noindent Solving these recurrence relations (with appropriate
initial conditions), we get

$$a_n = (\frac{1}{2}+\frac{1}{\sqrt{8}})(3+\sqrt{8})^n +
(\frac{1}{2}-\frac{1}{\sqrt{8}})(3-\sqrt{8})^n,  \hskip .3cm b_n =
\frac{1}{\sqrt{8}}(3+\sqrt{8})^n
-\frac{1}{\sqrt{8}}(3-\sqrt{8})^n.$$

Thus, $a_n$ is the largest entry of $(C(2))^n$, and we conclude that
no  entry of $(C(2))^n$ is larger than $p$ as long as $n <
\log_{_{3+\sqrt{8}}} p$. Since $C(2)=A(2)B(2)$ is a product of two
generators, $(C(2))^n$ has length  $2n$ as a word in the  generators
$A(2)$  and  $B(2)$. Therefore, no two positive words of length $\le
m$ in the generators $A(2)$  and  $B(2)$ (considered as matrices
over $\F_p$) can be equal as long as

$$m < 2 \log_{_{3+\sqrt{8}}} p = \log_{_{\sqrt{3+\sqrt{8}}}} p,$$

\noindent so we have the following

\begin{corollary}
There are no collisions of the form  $u(A(2), B(2)) = v(A(2), B(2))$
if positive words $u$ and  $v$ are of length less than
$\log_{_{\sqrt{3+\sqrt{8}}}} p.$ In particular, the girth of the
Cayley graph of the semigroup generated by $A(2)$ and  $B(2)$
(considered as matrices over $\F_p$) is at least
$\log_{_{\sqrt{3+\sqrt{8}}}} p.$

\end{corollary}

The base of the logarithm here is $\sqrt{3+\sqrt{8}} \approx 2.4$.
Thus, for example, if $p$ is on the order of $2^{256}$,  then there
are no collisions of the form  $u(A(2), B(2)) = v(A(2), B(2))$ if
positive words $u$ and  $v$ are of length less than 203.

We also note, in passing, that our Proposition \ref{alternate} also
holds without the assumption on the  words $w_n(a, b)$ to be
positive if we consider the absolute values of the matrix entries
and their sums. Our lower bound on the girth of the Cayley graph of
the {\it group} generated by $A(2)$ and $B(2)$ therefore improves
(in this particular case) the lower bound given in \cite{H}, where
the base of the logarithm in the lower bound is 3.

\subsection{Powers of $C(3)=A(3)B(3)$}

The matrix $\cthree$ is $\begin{pmatrix}10&3\\3&1\end{pmatrix}$. If
we let $\cthreen=\begin{pmatrix}a_n & b_n\\ c_n & d_n\end{pmatrix}$,
and use the fact that $\cthreen=\cthree\cdot \cthree^{n-1}$, we get
the following recurrence relations:
\begin{align*}
a_n &= 10a_{n-1} + 3b_{n-1}\\
b_n &= 3a_{n-1} + b_{n-1}=c_n\\
d_n &= a_{n-1}
\end{align*}
Combining these recurrence relations for $a_n$ and $b_n$, we get
$$
a_n=11a_{n-1}-a_{n-2},\quad b_n=11b_{n-1}-b_{n-2}.
$$
Solving these recurrence relations with the initial conditions
$a_1=10, a_2=109$ and $b_1=3, b_2=33$, we get

\begin{gather*}
a_n=
\Bigl(\frac{9}{2\sqrt{117}}+\frac12\Bigr)\Bigl(\frac{11+\sqrt{117}}{2}\Bigr) +
\Bigl(\frac12-\frac{9}{2\sqrt{117}}\Bigr)\Bigl(\frac{11-\sqrt{117}}{2}\Bigr),\\
b_n=\frac{3}{\sqrt{117}}\Bigl(\frac{11+\sqrt{117}}{2}\Bigr) -
\frac{3}{\sqrt{117}}\Bigl(\frac{11-\sqrt{117}}{2}\Bigr).
\end{gather*}

From this we see that $a_n$ is the largest entry of $\cthreen$, so
no entry of $\cthreen$ is larger than $p$ if
$n<\log_{\frac{11+\sqrt{117}}{2}}p$. Since $\cthree$ is a product of
two generators, $\athree$ and $\bthree$, we have:

\begin{corollary}
There are no collisions of the form  $u(A(3), B(3)) = v(A(3), B(3))$
if positive words $u$ and  $v$ are of length less than
$2\log_{\frac{11+\sqrt{117}}{2}}p =
\log_{\sqrt{_{\frac{11+\sqrt{117}}{2}}}}p.$ In particular, the girth
of the Cayley graph of the semigroup generated by $A(3)$ and  $B(3)$
(considered as matrices over $\F_p$) is at least
$\log_{\sqrt{_{\frac{11+\sqrt{117}}{2}}}}p.$

\end{corollary}

The base of the logarithm here is $\sqrt{_{\frac{11+\sqrt{117}}{2}}}
\approx  3.3$. For example, if $p$ is on the order of $2^{256}$,
then there are no collisions of the form  $u(A(2), B(2)) = v(A(2),
B(2))$ if positive words $u$ and  $v$ are of length less than 149.

\section{Conclusions}

We have analyzed the girth of the Cayley graph of the group and the
monoid generated by pairs of matrices $A(k)$ and $B(k)$ (considered
over $\F_p$), for various $k \ge 1$. Our conclusions are:

$\bullet$ The ``lifting attack" by Tillich  and Z\'emor
\cite{TZ_attack} that produces explicit relations of length $O(\log
p)$ in the monoid generated by $A(1)$ and $B(1)$, can be used in
combination with an old result by Sanov \cite{Sanov} and some
results from the theory of automatic groups \cite{Epstein} to
efficiently produce explicit relations of length $O(\log p)$ in the
{\it group} generated by $A(2)$ and $B(2)$.

$\bullet$  Generically, relations produced by this method will
involve negative as well as positive powers of $A(2)$ and $B(2)$,
and therefore will {\it not} produce collisions for the
corresponding hash function.

$\bullet$ In the absence of a result for $A(3)$ and $B(3)$ similar
to Sanov's result for $A(2)$ and $B(2)$, at this time there is no
known efficient algorithm for producing explicit relations of length
$O(\log p)$ even in the {\it group} generated by $A(3)$ and $B(3)$,
let alone in the monoid generated by this pair. At the same time,
such relations do exist by the pigeonhole principle.

$\bullet$ We have computed an explicit lower bound for the length of
relations in the monoid generated by $A(2)$ and $B(2)$; the lower
bound is $\log_b p,$ where the base $b$ of the logarithm is
approximately $2.4$. For the monoid generated by $A(3)$ and $B(3)$,
we have a similar lower bound, with base $b$ of the logarithm
approximately equal to $3.3$.

$\bullet$ We conclude that at this time, there are no known attacks
on hash functions corresponding to the pair $A(2)$ and $B(2)$ or
$A(3)$ and $B(3)$ and therefore no visible threat to their security.

\bigskip

\noindent {\it Acknowledgement.} We are grateful to Ilya Kapovich
for helpful comments, in particular for pointing out the relevance
of some results from \cite{Epstein} to our work. We are also
grateful to Harald Helfgott for useful discussions.


\baselineskip 11 pt


\begin{thebibliography}{ABC}


\bibitem{BG}
J. Bourgain, A. Gamburd, {\it Uniform expansion bounds for Cayley
graphs of $SL_2({\mathbf F}_p)$}. Ann. of Math. (2) {\bf 167}
(2008),  625--642.

\bibitem{CLS}
S. Contini, A. K. Lenstra and R. Steinfeld,  {\it VSH, an Efficient
and Provable Collision Resistant Hash Function}, in: Eurocrypt 2006,
Lecture Notes  Comp. Sci. {\bf 4004} (2006), 165--182.

\bibitem{Epstein}
D. B. A. Epstein, J.  Cannon, D. F. Holt, S. V. F.  Levy, M. S.
Paterson, W. P. Thurston,  {\it  Word processing in groups.} Jones
and Bartlett Publishers, Boston, MA, 1992.


\bibitem{H}
H. A. Helfgott,{\it Growth and generation in $SL_2(\Z/p\Z$)} Ann. of
Math. (2) {\bf 167} (2008),  601--623.

\bibitem{Larsen}
M. Larsen, {\it Navigating the Cayley graph of $SL_2({\mathbf
F}_p)$}, Int. Math. Res. Notes  {\bf 27} (2003), 1465-–1471.

\bibitem{Lub} A. Lubotzky, {\it Discrete groups, expanding graphs
and invariant measures}, Progress in Mathematics {\bf 125},
Birkh\"auser Verlag, Basel, 1994.

\bibitem{Marg} G. A. Margulis, {\it Explicit constructions of
concentrators}, Problems of Information Transmission {\bf 9} (1973),
no. 4, 325--332.

\bibitem{MOV}
A. Menezes, P. van Oorschot and S. Vanstone, {\it Handbook of
Applied Cryptography}, CRC Press, 1997.


\bibitem{Petit_thesis}
C. Petit, {\it On graph-based cryptographic hash functions}, PhD
thesis, 2009.

\bibitem{Petit_Rubik}
C. Petit and J.-J. Quisquater, {\it Rubik's for cryptographers},
Notices  Amer. Math. Soc. {\bf  60} (2013), 733--739.

\bibitem{Sanov}
I. N. Sanov, {\it A property of a representation of a free group}
(Russian), Doklady Akad. Nauk SSSR (N. S.)  {\bf 57} (1947),
657--659.

\bibitem{TZ_attack}
J.-P. Tillich and G. Z\'emor,  {\it Group-theoretic hash functions},
in Proceedings of the First French-Israeli Workshop on Algebraic
Coding, Lecture notes  Comp. Sci. {\bf  781} (1993),   90--110.

\bibitem{TZ}
J.-P. Tillich and G. Z\'emor, {\it Hashing with $SL_2$}, in CRYPTO
1994, Lecture Notes  Comp. Sci. {\bf 839} (1994), 40--49.



\end{thebibliography}
\end{document}